\documentclass[a4paper, 12pt]{amsart}
\usepackage[ngerman, english]{babel}
\usepackage{amsfonts,amsmath}
\usepackage{amsthm}
\usepackage{txfonts}
\usepackage{faktor}
\usepackage[arrow, matrix, curve]{xy}

\newenvironment{remark}[1][Remark]{\begin{trivlist}
\item[\hskip \labelsep {\bfseries #1}]}{\end{trivlist}}

\newtheorem{theorem}{Theorem}[section]

\newtheorem{proposition}[theorem]{Proposition}

\newtheorem{definition}[theorem]{Definition}
\newtheorem{conjecture}[theorem]{Conjecture}

\newcommand{\defi}{\coloneqq}
\newcommand{\C}{\mathbb C}
\newcommand{\R}{\mathbb{R}}

\newcommand{\N}{\mathbb{N}}

\newcommand{\W}{\mathcal{W}}

\newcommand{\abs}[1]{\lvert#1\rvert}
\newcommand{\Spn}{\mathbb{S}^n}
\newcommand{\norm}[1]{\ensuremath{\left\|#1\right\|}}
\newcommand{\Sp}{\mathbb{S}}

\title[Klein bottles]{A note on Willmore minimizing Klein bottles in Euclidean space}

\author[J. Hirsch]{Jonas Hirsch}
\address[Jonas Hirsch]{Karlsruhe Institute of Technology, Institute for Analysis, Englerstr. 2, 76131 Kalrsruhe, Germany}
\email{jonas.hirsch@kit.edu}

\author[E. M\"ader-Baumdicker]{Elena M\"ader-Baumdicker}
\address[Elena M\"ader-Baumdicker]{Karlsruhe Institute of Technology, Institute for Analysis, Englerstr. 2, 76131 Kalrsruhe, Germany}
\email{elena.maeder-baumdicker@kit.edu}

\begin{document}
\begin{abstract}
 We show that $\varphi\circ \tilde\tau_{3,1}:K\to\R^4\times \{0\}^{n-4}$ is the unique minimizer among immersed Klein bottles in its conformal class, where $\varphi:\mathbb S^4\to\R^4$ is a stereographic projection and $\tilde\tau_{3,1}$ is the bipolar surface of Lawson's $\tau_{3,1}$-surface \cite{Lawson}. 
%  Furthermore, we prove that $\tilde\tau_{3,1}$ is the only minimally immersed Klein bottle into $\mathbb S^n$ for some $n\in\N$ where the coordinate functions are first eigenfunctions of the Laplacian. 
 We conjecture that $\varphi\circ\tilde\tau_{3,1}$ is the unique minimizer among immersed Klein bottles into $\R^n$, $n\geq 4$, whose existence the authors and P.\ Breuning proved in  \cite{we}.
\end{abstract}

\maketitle

%TODO mögliche Dinge zum reinschreiben: Penskoi Parametrisierung, was ist eine bipolare Fläche, da gibt es auch explizite Parametrisierungen

\section{Introduction}
Let $M$ be a closed manifold of dimension two.
The Willmore energy of an immersed surface $f:M\to\R^n$, $n\geq 3$, is defined by
\begin{align*}
 \W(f)\defi \frac{1}{4} \int_M |H|^2 d\mu_g,
\end{align*}
where $H$ is the mean curvature vector of the surface, $g\defi f^\sharp \delta_{\R^n}$ the induced metric on $M$ and $d\mu_g$ the induced area measure. \\
Willmore \cite{Willmore} proved 1965 that $\W(f)\geq 4\pi$ holds for any closed surface in $\R^3$ with equality for round spheres.  %TODO: nur codim 1?
He also computed the Willmore energy of certain tori and found out that the minimum of the Willmore energy among these tori is attained by a stereographic projection of the Clifford torus (with energy $2\pi^2$). 
%This surface can be parametrized by $$\Sigma=\left\{\left((\sqrt{2} + \cos x)\cos y, (\sqrt{2} + \cos x)\sin y,\sin x\right)\in\R^3 : x, y\in\R\right\}.$$ 
He conjectured that every orientable surface in $\R^3$ of genus greater than zero satisfies $$\W(f)\geq 2\pi^2.$$ This \emph{Willmore conjecture} was proved by Marques and Neves \cite{MarquesNevesWillmore}. Before the proof of Marques and Neves appeared a lot of partial results were obtained concerning the Willmore conjecture, see \cite{MarquesNeves} and the references therein. \\%TODO evtl mehr dazu schreiben...Ndiaye Schätzle usw?
For non-orientable surfaces the number of results concerning the Willmore energy are quite limited. Li and Yau proved in \cite{LiYau} that $W(f)\geq 6\pi$ for any immersed $\R P^2$ in $\R^n$, $n\geq 4$, with equality if and only if $f:\R P^2\to\R^4$ is the Veronese embedding. As there always is a triple point for immersed $\R P^2$ in $\R^3$ \cite{Banchoff} an inequality from \cite{LiYau} gives $\W(f)\geq 12 \pi$ with equality for example for Boy's surface \cite{Kusner, Bryant}.\\
In \cite{we}, the authors proved together with P.\ Breuning that the infimum among all immersed Klein bottles in $\R^n$, $n\geq 4$ is attained by a smooth embedding. The value of this minimum is strictly less than $8\pi$. In this paper, we prove that it is less or equal to $6 \pi \operatorname E \left(\frac{2\sqrt{2}}{3}\right)\approx 6.682 \pi$, where $\operatorname E(.)$ is the complete elliptic integral of second kind. Our first result is the following:

\begin{theorem} \label{main1}
 Consider the bipolar surface of Lawson's $\tau_{3,1}$ torus and denote it by $\tilde \tau_{3,1}$. It is known that the surface $\tilde \tau_{3,1}$ is a minimally embedded Klein bottle in $\mathbb S^4$ \cite{Lapointe}. Let $\varphi:\mathbb S^4\to\R^4$ be a stereographic projection. Then we have that $\varphi\circ \tilde \tau_{3,1}:K\to\R^4\times \{0\}^{n-4}$, $n\geq 4$, is the minimizer of the Willmore energy in its conformal class, i.e.\ we have that
 \begin{align}\label{ungl}
  \W(f)\geq 6 \pi \operatorname E \left(\frac{2\sqrt{2}}{3}\right)\approx 6.682 \pi
 \end{align}
for every immersed Klein bottle $f:K\to \R^n$, $n\geq 4$, that is conformal to $\varphi\circ\tilde \tau_{3,1}$. Here, $\operatorname E(.)$ is the complete elliptic integral of second kind.
Furthermore, equality in (\ref{ungl}) for an immersed Klein bottle $f$ conformal to $\varphi\circ \tilde \tau_{3,1}$ implies that $f$ is the surface $\varphi\circ \tilde \tau_{3,1}$ up to conformal diffeomorphisms of $\R^n$.
\end{theorem}

In the proof, we use the conformal volume studied by Li and Yau \cite{LiYau} and a result by Jakobson, Nadirashvili and Polterovich \cite{Jakobson} who found out that $\tilde\tau_{3,1}$ is embedded by first eigenfuctions of the Laplacian. El Soufi, Giacomini and Jazar \cite{ElSoufi2} proved that the metric on $\tilde\tau_{3,1}$ is the only metric on a Klein bottle that is critical for the first eigenvalue. This implies
% Using a uniqueness result of El Soufi, Giacomini and Jazar \cite{ElSoufi2} we get the following result:

\begin{theorem}[El Soufi, Giacomini, Jazar \cite{ElSoufi2}] \label{main2}
 The only minimal Klein bottle immersed into $\mathbb S^n$ (for some $n\in\N$) by the first eigenfuctions of the Laplacian is Lawson's bipolar surface $ \tilde \tau_{3,1}$.
\end{theorem}

\begin{remark}
 The analog result for tori was proved by Montiel and Ros \cite{Montiel}: The only minimal torus immersed into $\mathbb S^3$ by the first eigenfunctions is the Clifford torus. Note that the situation for tori is somehow different. There is another torus immersed by first eigenfunctions, the standard embedding of the equilateral torus, into $\Sp^5$. This surface has Willmore energy $\frac{4\pi^2}{\sqrt{3}}$. Both metrics are extremal for the first eigenvalue on a torus (for the definition see Section~\ref{Section2}). The metric on the equilateral torus is the maximum point of $\lambda_1 V$ \cite{Nadi}, and the metric on the Clifford torus is a saddle point of $\lambda_1 V$. These two surfaces are the only tori that are immersed by first eigenfunctions into some $\Sp^n$ \cite{ElSoufi3}.
\end{remark}

These theorems let us think that $ \tilde \tau_{3,1}$ is ``the best'' Klein bottle concerning its Willmore energy.

\begin{conjecture}\label{conj1}
 The unique minimizer among all immersed Klein bottles in $\R^n$, $n\geq 4$, is  $\varphi\circ \tilde\tau_{3,1}$. 
 %If the conjecture is true then the infimum among all immersed Klein bottles in $\R^4$ would be 
% $$6 \pi \operatorname E \left(\frac{2\sqrt{2}}{3}\right)\approx 6.682 \pi,$$
% where $\operatorname E(.)$ is the complete elliptic integral of second kind.
%  is defined by
%  \begin{align*}
%   T_{1,0,2}(x,y)= 
%   \begin{pmatrix}
%    \frac{1}{\sqrt{2}} \sin x\sin y \\ \frac{1}{\sqrt{2}} \cos x\sin y \\ \sqrt{\frac{5}{8}} \cos y \\ \sqrt{\frac{3}{8}} \sin 2 x \sqrt{1 + \frac{1}{3} \sin^2 y}\\ \sqrt{\frac{3}{8}} \cos 2 x \sqrt{1 + \frac{1}{3} \sin^2 y}
%   \end{pmatrix}\, .
%  \end{align*}
%  This surfaces was constructed in ??Penskoi and it was shown to be isometric to the bipolar Lawson Klein bottle $\tilde\tau_{3,1}$ ??Lawson.

\end{conjecture}

In Section~\ref{Section2} we recall some definitions and prove Theorem~\ref{main1} and Theorem~\ref{main2}.

\section*{Acknowledgment}
We want to thank Tobias Lamm for very good hints during coffee time.

\section{Extremal metrics on a Klein bottle and the Willmore energy} \label{Section2}

Before we prove Theorem~\ref{main1} and Theorem~\ref{main2} we recall some definitions and required results. Lawson constructed in his work \cite{Lawson} important families of minimal immersions into $\Sp^3$. One of them are the tori (and Klein bottles) $\tau_{m,k}$. He gave an explicit parametrization: To each unordered pair of positive integers $\{m,k\}$ with $(m,k)=1$  there corresponds a distinct compact minimal surface $\tau_{m,k}$ given by 
\begin{align}
\Psi_{m,k} & :\R^2\to\Sp^3,\\
 \Psi_{m,k}(x,y) &= (\cos mx \cos y, \sin mx \cos y,\cos kx \sin y,\sin kx \sin y).
\end{align}
If $m$ and $k$ are odd, then $\tau_{m,k}$ is a torus parametrized over $\C / \Gamma$, where $\Gamma$ is the lattice generated by $(\pi,\pi)$ and $(\pi,-\pi)$. If $m$ is even, then $\tau_{m,k}$ is invariant under the actions $ (x,y)\mapsto (x + \pi, -y)$ and $(x,y)\mapsto (x,y + 2\pi)$. If $k$ is even then it is invariant under $(x,y)\mapsto (x + \pi, \pi-y) $ and $(x,y)\mapsto (x,y + 2\pi)$. In both cases the surface is a Klein bottle.\\
Lawson studied in \cite[Chapter~11]{Lawson} bipolar surfaces of his $\tau_{m,k}$-surfaces. Let $M$ be a manifold of dimension two and $\psi:M\to \Sp^3\subset \R^4$ a minimal immersion. Denote by $\psi^\ast:M\to\Sp^3\subset \R^4$ its associated Gauss map. Then
$\tilde\psi:M\to\Sp^5\subset \R^6$,
\begin{align*}
 \tilde\psi\defi \psi\wedge\psi^\ast
\end{align*}
is a minimal immersion into $\Sp^5$. Lawson showed that the image of $\tilde\tau_{m,k}$ actually lies in a plane intersected with $\Sp^5$ if $\{m,k\}\not =\{1,1\}$. Thus, $\tilde\tau_{m,k}$ can be regarded as a an immersion into $\Sp^4$. The surface $\tilde\tau_{1,1}$ lies in $\Sp^3$. Lapointe studied these surfaces in \cite{Lapointe}. He showed that $\tilde\tau_{m,k}$ is a Klein bottle if $mk=3\mod 4$. If $mk=0\mod 2$ or $mk=1\mod 4$ then $\tilde\tau_{m,k}$ is a torus.

For our work it is important to get the connection between minimal surfaces in $\Sp^n$ and so called extremal metrics. We refer to the article \cite{Penskoi1} for an introduction and an overview of recent progress in this field. If $M$ is a closed $m$-dimensional manifold and $g$ a Riemannian metric on $M$. The associated Laplacian $\Delta: C^\infty(M)\to\C^\infty(M)$ is given in local coordinates by
\begin{align*}
 \Delta f = -\frac{1}{\sqrt{\det g}} \partial_i \left(\sqrt{\det g}\, g^{ij} \partial_j f\right).
\end{align*}
The spectrum of $\Delta$ is discrete, non-negative and tends to infinity. We denote it by
\begin{align*}
 \sigma(M,g) =\{0<\lambda_1(g)\leq \lambda_2(g)\leq \cdots \leq \lambda_k(g)\leq \cdots\},
\end{align*}
where the eigenvalues are repeated according to their multiplicities. Note that $g\mapsto \lambda_k(g) V(M, g)^{\frac{2}{m}}$ is scaling invariant, where $V(M,g)$ is the volume of $M$ with respect to $g$. Denote by $\mathcal M(g)$ the set of all Riemannian metrics on $M$ having the same volume as $g$. The map $g\mapsto \lambda_k(g) V(M,g)^{\frac{2}{m}}$ is not differentiable on $\mathcal M(g)$ but there is a natural definition of extremality, see for example \cite{ElSoufi4}:
\begin{definition}
 Let $g_t$ be an volume preserving, analytic deformation of $g$ with $g_0=g$.  Then the map $t\mapsto \lambda_k(g_t)$ admits left and right derivatives at $t=0$ (see \cite{ElSoufi4}). If these derivatives satisfy 
 \begin{align*}
  \frac{d}{d t}\lambda_k(g_t)\big|_{t=0^-} \times \frac{d}{d t}\lambda_k(g_t)\big|_{t=0^+}\leq 0
 \end{align*}
then the metric is called \emph{extremal}. 
\end{definition}
For the first (non-trivial) eigenvalue $\lambda_1$ the condition to be extremal can also be formulated as $\lambda_1(M,g_t)\leq \lambda_1(M,g) + o(t)$ as $t\to 0$ for any volume preserving analytic deformation $g_t$ of $g$. The connection to minimal immersions into spheres is the following:

% \begin{theorem}
% Let $f:M\to\Sp^n\subset\R^{n+1}$ be a minimal immersion into $\Sp^n$. Then the restrictions $x^1|_M,..., x^{n+1}|_M$ of the coordinate functions to $M$ are eigenfunctions of the Laplacian $\Delta$ on $M$ with the metric induced by the $f$. The eigenfunctions correspond to the eigenvalue $\lambda= \dim M$.
% \end{theorem}

\begin{theorem}[Nadirashvili \cite{Nadi}, El Soufi, Ilias \cite{ElSoufi4}]
Let $f:M\to\Sp^n\subset\R^{n+1}$ be a minimal immersion into $\Sp^n$. Then the metric on $M$ induced by $f$ is extremal for $ \lambda_{N(\dim M)}$, where $N(\lambda)=\sharp \{\lambda_k: \lambda_k< \lambda\}$ is the Weyl eigenvalues counting function.\\
On the other hand: if a metric $g$ is extremal for $\lambda_k$ and $\lambda_k <\lambda_{k + 1}$ or $\lambda_{k-1}<\lambda_k$ then there exists a finite family $\{u_1,...,u_d\}$ of eigenfunctions associated with $\lambda_k$ such that $$\sum_{i\leq d}d u_i \otimes du_i =g.$$
This implies that the map $u\defi (u_1,...,u_d):(M,g)\to\R^d$ is an isometric immersions whose image is a minimally immersed submanifold of the sphere $\Sp^{d-1}\left(\sqrt{\frac{\dim M}{\lambda_k}}\right)$ of radius $\sqrt{\frac{\dim M}{\lambda_k}}.$
\end{theorem}

As we use the results of \cite{LiYau} we recall their terminology:
\begin{definition}
 Let $M$ be a compact $m$-dimensional manifold that admits a conformal map $\phi:M\to\Sp^n$. Denote by $G$ the group of conformal diffeomorphisms on $\Sp^n$. Then the $n$-\emph{conformal volume of }$\phi$ is defined as
 \begin{align*}
  V_c(n,\phi)\defi \sup_{g\in G} V(M, g\circ \phi).
 \end{align*}
 The $n$-\emph{conformal volume of} $M$ is defined as
 \begin{align*}
  V_c(n,M)\defi \inf_{\phi} V_c(n,\phi).
 \end{align*}
The \emph{conformal volume} of $M$ is
\begin{align*}
 V_c(M)\defi \inf_{n\geq 2} V_c(n,M).
\end{align*}
\end{definition}

We need the following result of Jakobson, Nadirashvili and Polterovich \cite[Theorem~1.3, Theorem~1.4]{Jakobson}:

\begin{theorem}[Jakobson, Nadirashvili, Polterovich  \cite{Jakobson}] \label{thmtau}
~\\
 A metric of revolution
 \begin{align*}
  g_0= \frac{9 + (1 + 8 \cos^2 v)^2}{1 + 8 \cos^2 v}\left(d u^2 + \frac{dv^2}{1 + 8 \cos^2 v}\right),
 \end{align*}
$0\leq u<\frac{\pi}{2}$, $0\leq v<\pi$, is an extremal metric for the first eigenvalue on a Klein bottle $K$. The surface $(K,g_0)$ admits a minimal isometric embedding into $\mathbb S^4$ by the first eigenfunctions. The first eigenvalue of the Laplacian for this metric has multiplicity $5$ and satisfies the equality
\begin{align*}
 \lambda_1 V(K,g_0) = 12 \pi \operatorname E \left(\frac{2\sqrt{2}}{3}\right),
\end{align*}
where $V(K,g_0)$ is the volume of the surface $(K,g_0)$. Furthermore, the surface $(K,g_0)$ is the bipolar surface of Lawson's $\tau_{3,1}$-torus.
\end{theorem}
Now we prove the first part of Theorem~\ref{main1}.

\begin{theorem}\label{thm.lower bound in the conformal class}
 Let $f:K\to\R^n$ be an immersed Klein bottle such that $g=f^\sharp \delta^{\R^n}$ is conformal to $\left(\varphi\circ \tilde\tau_{3,1}\right)^\sharp \delta_{ \R^n}$, where $\tilde \tau_{3,1}$ is the bipolar surface of Lawson's $\tau_{3,1}$-surface and $\varphi:\mathbb S^4\to\R^4$ a stereographic projection. Then 
 \begin{align*}
  \W(f)\geq 6 \pi \operatorname E \left(\frac{2\sqrt{2}}{3}\right)\approx 6.682 \pi.
 \end{align*}
Equality holds for $\varphi\circ \tilde\tau_{3,1}$. 
\end{theorem}

\begin{proof}
As the surface $\tilde\tau_{3,1}$ admits a minimal isometric embedding into $\mathbb S^4$ due to Jakobson, Nadirashvili and Polterovich (Theorem~\ref{thmtau}) we get by \cite[Corollary~4]{LiYau} that 
$$V_c(K)=V_c(4,K)=V(K,g_0),$$
where we used the expressions defined by Li and Yau: \\
$V_c(4,K)$ is the $4$-conformal volume of $K$, \\
$V_c(K)$ is the conformal volume of $K$ and \\
$V(K,g_0)$ is the volume of the surface. \\
As $\tilde\tau_{3,1}$ is minimally embedded into $\mathbb S^4$ by the first eigenfunctions we have that $\lambda_1=2$ \cite[Proposition~1.1]{ElSoufi3}. We use \cite[Lemma~1]{LiYau}, i.e.
\begin{align*}
 \W(f)\geq V_c(n,K)
\end{align*}
and the fact that the conformal volume of two immersions coincides when they are conformal to each other and get
\begin{align*}
  2 \W(f)\geq 2 V_c(n,K) = 2 V_c(K) = 2 V(K,g_0) = 12 \pi \operatorname E \left(\frac{2\sqrt{2}}{3}\right).
\end{align*}

\end{proof}
In the following we want to show uniqueness of Lawson's $\tilde\tau_{3,1}$ Klein bottle as a Willmore minimizer in its conformal class (which is the second part of Theorem~\ref{main1}). To do so we need a generalization of a result of El Soufi and Ilias \cite[Proposition 3.1. (i)]{ELSoufi}.

\begin{proposition}\label{prop:rigidity1}
Let $f: M \to \mathbb S^n$ be an $m$-dimensional minimal immersion,  $g\defi f^\sharp\delta_{\mathbb S^n}$ and $M \neq \mathbb S^m$. Furthermore, let $f$ satisfy 
\begin{equation}\label{eq.rigidity1} 
m V(M,g)^{\frac 2m} 
%\sup_{\substack{ \tilde{f} \text{ conformal to } f\\ \tilde{g}
%= \tilde{f}^\sharp \delta_{\mathbb S^n}}} \lambda_1(\tilde{g}) V(M, \tilde{g})^{\frac 2m}
=\lambda_1(g_0) V(M, g_0)^{\frac 2m} \end{equation}
for some smooth Riemannian metric $g_0$. Then we have $g=k g_0$ for some $k>0$. Moreover $f$ is given by a subspace of the first eigenspace. 
\end{proposition}

\begin{proof}
Scaling $g_0$ we may assume without loss of generality that $\lambda_1(g_0)=m$. Since $g$ and $g_0$ are conformal i.e. $g_0=e^{2u} g$ for some smooth $u: M \to \R$ the first part of the proposition is proved showing that $u\equiv 0$.\\
Let $G$ denote the group of conformal diffeomorphisms of $\mathbb S^n$. Since $f$ is assumed to be minimally immersed by \cite[\textsection 3 Proposition 1]{LiYau} for $m=2$ and \cite[Theorem 1.1]{ELSoufi} for $m \ge 3$ we have
\begin{equation}\label{eq:ineq1} V(M, g) = \sup_{\gamma \in G} V(M, (\gamma \circ f)^\sharp \delta_{\Spn}) \ge V_c(n,M).\end{equation}
% Moreover, \cite[\textsection 2 Theorem 1]{LiYau} for $m=2$ and \cite[Proposition 3.1]{ELSoufi} implies that 
% \[ m V(M,g)^{\frac 2m} \ge \sup_{\substack{ \tilde{f} \text{ conformal to } f\\ \tilde{g}= \tilde{f}^\sharp \delta_{\mathbb S^n}}} \lambda_1(\tilde{g}) V(M, \tilde{g})^{\frac 2m}.\]
As shown in the proof of \cite[\textsection 2 Theorem 1]{LiYau} there exists $\gamma \in G$ such that 
\[ \int_{M} X\circ \gamma \circ f \, d\mu_{g_0} =0\]
where $X=(X^1, \dotsc , X^{n+1})$ are the coordinate functions of $\R^{n+1}$.\\
Set $\hat{f}:=\gamma\circ f$ and $\hat{f}^i=X^i\circ \gamma \circ f$. By the variational characteristic of $\lambda_1$ and H\"olders inequality in the case of $m >2$ we have
\[ \lambda_1(g_0)V(M,g_0) \le \int_M \abs{D\hat{f}}_{g_0}^2 \, d\mu_{g_0} \le \left( \int_{M} \abs{D\hat{f}}_{g_0}^m \, d\mu_{g_0}\right)^{\frac 2m} V(m,g_0)^{1-\frac 2m}. \]
The first factor is conformally invariant. Thus, by \eqref{eq:ineq1} we deduce
\begin{equation}\label{eq:ineq2} \int_{M} \abs{D\hat{f}}_{g_0}^m \, d\mu_{g_0} = \int_{M} \abs{D\hat{f}}_{g}^m \, d\mu_{g} = m^{\frac m2} V(M, (\gamma \circ f)^\sharp \delta_{\Spn}) \le m^{\frac m2} V(M, g).\end{equation}
Combining all inequalities we get
\[ \lambda_1(g_0)V(M,g_0)^{\frac 2m} \le \int_{M} \abs{D\hat{f}}_{g_0}^2 \, d\mu_{g_0} V(M,g_0)^{-1+\frac 2m} \le m V(M, g)^{\frac{2}{m}}.\]
By assumption \eqref{eq.rigidity1} we must have equality at all steps. Equality in the first implies that $\hat{f}^i$ are eigenfuctions of $\Delta_{g_0}$ to $\lambda_1(g_0)=m$. Equality in the second implies that $V(M, (\gamma \circ f)^\sharp \delta_{\Spn}) =  V(M, g)$ and since $f$ is a minimal immersion we deduce by \cite[Theorem 1.1]{ELSoufi} that $\gamma \in O(n+1)$ and $\hat{f}$ is a minimal immersion and $g= \hat{f}^\sharp \delta_{\Spn}$ as well. By \cite[Theorem 3]{Taka} this implies that $\hat{f}^i$ are eigenfunctions of $\Delta_{g}$ to some eigenvalue $\lambda_k(g)=m$. Since $g_0=e^{2u}g$ one has for any $\varphi: M \to \R$ in local coordinates that $ \Delta_{g_0} \varphi = e^{-2u} \left( \Delta_{g} \varphi + (m-2) g^{kl} \partial_k u \partial_l \varphi\right)$. Apply this to any $\hat{f}^i$ we deduce that 
\[ m \hat{f}^i (e^{2u}-1) = (m-2) g^{kl} \partial_k u \partial_l \hat{f}^i.\]
If $m=2$ we deduce directly $u\equiv 0$ choosing a nontrivial $\hat{f}^i$. If $m >2$ we argue as follows: multiply the above by $\hat{f}^i$ sum in $i$ and use $\sum_{i=1}^{n+1} (\hat{f}^i)^2 =1 $ hence $ m (e^{2u}-1) = 0$.
The remaining conclusions follow easily.
\end{proof}

\begin{proposition}[Uniqueness of $\tilde\tau_{3,1}$ in its conformal class]\label{prop.rigidity}
~\\
If $f: K \to \R^n$ is a Willmore minimizer in the conformal class of Lawson's bipolar $\tilde\tau_{3,1}$-surface i.e. $f$ conformal to $\varphi\circ \tilde{\tau}_{3,1}$ then $f$ coincides with $\varphi\circ\tilde{\tau}_{3,1}$ up to a conformal diffeomorphism of $\R^n$. 
\end{proposition}

\begin{proof}
Let $f:K\to\R^n$ be a minimizer of the Willmore energy in the conformal class of $\varphi\circ \tilde\tau_{3,1}$. 
Theorem~\ref{thm.lower bound in the conformal class} implies that \[\W(f)=\W(\varphi\circ \tilde{\tau}_{3,1})=V_c(K, \varphi\circ \tilde{\tau}_{3,1})=V_c(K,f).\]
We apply \cite[\textsection 5 Lemma 1]{LiYau} to deduce that there exists $\gamma \in G$ such that $\hat{f}:=\gamma \circ \varphi^{-1} \circ f$ is a minimal surface in $\Spn$. Lawsons $\tilde\tau_{3,1}$-surface is extremal for the first eigenvalue, compare Theorem~\ref{thmtau}. Hence we are in the situation of Proposition \ref{prop:rigidity1} and deduce that $g:= \hat{f}^\sharp \delta_{\Spn}$ coincides with $ g_0:= \tilde{\tau}_{3,1}^\sharp \delta_{\Spn}$.
Fix $O \in O(n+1)$ diagonalising and ordering the symmetric positiv definite matrix $a^{ij}:= \int_{M} \hat{f}^i\hat{f}^j \, d\mu_{g}$. Considering $O\hat{f}$ instead of $\hat{f}$ itself we may assume that that the different $\hat{f}^i$ are pairwise orthogonal and satisfy $\int_M (\hat{f}^i)^2 \, d\mu_{g} \ge \int_M (\hat{f}^{i+1})^2 \, d\mu_{g}$ for all $i =1, \dotsc, n$. \\
The surface $\tilde{\tau}_{3,1}$ admits a minimal embedding into $\mathbb{S}^4$ by five eigenfunctions to $\lambda_1(g_0)$, compare Theorem~\ref{thmtau}. Hence $\hat{f}$ coincides with $\tilde{\tau}_{3,1}$ up to a element in $O(5)$ if we show that $\norm{\hat{f}^i}_{L^2} \neq 0 $ for $i \le 5 $ and $\norm{\hat{f}^i}_{L^2}=0$ for $i>5$. Since $\W(f)< 8\pi$, compare Theorem~\ref{thm.lower bound in the conformal class}, we have that $f$ and $\hat{f}$ are embeddings. A Klein bottle cannot be embedded into $\mathbb{S}^3$ hence we must have $\norm{\hat{f}^i}_{L^2}>0$ for at least $5$ indices. The multiplicity of $\lambda_1(g_0)$ is $5$, compare Theorem~\ref{thmtau} hence there can be at most $5$ nontrivial $\hat{f}^i$. 
\end{proof}

There is another indication for Conjecture~\ref{conj1} to be true. For that, we recall the following:

\begin{theorem}[El Soufi, Giacomini, Jazar, Theorem~1.1 in \cite{ElSoufi2}]\label{uniqueMetric}
~\\
 The Riemannian metric $g_0$ from \cite{Jakobson} (see Theorem~\ref{thmtau}) is the unique extremal metric of the functional $\lambda_1$ under volume preserving deformations of metrics on a Klein bottle $K$.
\end{theorem}

% \begin{corollary}
% The only minimal Klein bottle immersed into $\mathbb S^4$ by the first eigenfunctions of the Laplacian is Lawson's bipolar surface $\tilde\tau_{3,1}$.
% \end{corollary}

% \begin{remark}
%  The analog result for tori was proved by Montiel and Ros \cite{Montiel}: The only minimal torus immersed into $\mathbb S^3$ by the first eigenfunctions is the Clifford torus.
% \end{remark}

As a corollary we get another proof of Theorem~\ref{main2}:

\begin{proof}
(of Theorem~\ref{main2})
 Let $f:K\to \mathbb S^n$ be another minimal immersion where the coordinate functions are first eigenfunction of the Laplacian. Then $g\defi f^\sharp \delta_{\mathbb S^n}$ is extremal for $\lambda_1=2$, see \cite[Proposition~1.1]{ElSoufi3}.
 By Theorem~\ref{uniqueMetric} we have $g=g_0$. Thus, we have
 \begin{align*}
  2 V(K,g)= 2V(K,g_0)= \lambda_1(g_0) V(K,g_0)=12\pi E\left(\frac{2 \sqrt{2}}{3}\right).
 \end{align*}
 The proof of Proposition~\ref{prop.rigidity} shows that the coordinate functions of $ f$ coincide with those of  $\tilde\tau_{3,1}$ up to an element of $O(5)$.
\end{proof}

\bibliographystyle{plain}
\bibliography{Lit_2}

\end{document}